\documentclass{amsart}
\usepackage{amssymb, amsmath, mathrsfs, verbatim, url, bbm}

\theoremstyle{plain}
\newtheorem{theorem}{Theorem}
\newtheorem{lemma}[theorem]{Lemma}
\newtheorem{proposition}[theorem]{Proposition}
\newtheorem{corollary}[theorem]{Corollary}

\theoremstyle{definition}
\newtheorem{definition}[theorem]{Definition}
\newtheorem{question}{Question}

\newcommand{\cf}{\mathrm{cf}}
\newcommand{\cl}{\mathrm{cl}}
\newcommand{\ran}{\mathrm{ran}}
\newcommand{\dom}{\mathrm{dom}}
\newcommand{\ct}{{}^{<\omega}2}
\newcommand{\FR}{\mathrm{Cof}}
\newcommand{\FIN}{\mathrm{Fin}}
\newcommand{\Aa}{\mathcal{A}}
\newcommand{\BB}{\mathcal{B}}

\newcommand{\CC}{\mathcal{C}}
\newcommand{\DD}{\mathcal{D}}
\newcommand{\TT}{\mathcal{T}}

\newcommand{\FF}{\mathcal{F}}
\newcommand{\GG}{\mathcal{G}}

\newcommand{\XX}{\mathcal{X}}
\newcommand{\II}{\mathcal{I}}
\newcommand{\JJ}{\mathcal{J}}

\newcommand{\MM}{\mathcal{M}}
\newcommand{\UU}{\mathcal{U}}
\newcommand{\VV}{\mathcal{V}}
\newcommand{\PP}{\mathcal{P}}

\newcommand{\PPP}{\mathbb{P}}
\newcommand{\RRR}{\mathbb{R}}
\newcommand{\QQQ}{\mathbb{Q}}
\newcommand{\cccc}{\mathfrak{c}}
\newcommand{\pppp}{\mathfrak{p}}
\newcommand{\dddd}{\mathfrak{d}}

\begin{document}

\title{The topology of ultrafilters as subspaces of $2^\omega$}

\author{Andrea Medini}

\author{David Milovich}

\address{Math. Dept.\\
University of Wisconsin\\
480 Lincoln Dr.\\
Madison, WI 53706, USA}
\email{medini@math.wisc.edu}

\address{Texas A\&M International University\\ 
5201 University Blvd.\\ 
Laredo, TX 78041, USA}
\email{david.milovich@tamiu.edu}

%Keywords: ultrafilter; completely Baire; countable dense homogeneous; perfect set property; P-point; infinite power; Martin's Axiom

\date{November 27, 2011}

\begin{abstract}
Using the property of being completely Baire, countable dense homogeneity and the perfect set property we will be able, under Martin's Axiom for countable posets, to distinguish non-principal ultrafilters on $\omega$ up to homeomorphism. Here, we identify ultrafilters with subpaces of $2^\omega$ in the obvious way. Using the same methods, still under Martin's Axiom for countable posets, we will construct a non-principal ultrafilter $\UU\subseteq 2^\omega$ such that $\UU^\omega$ is countable dense homogeneous. This consistently answers a question of Hru\v{s}\'ak and Zamora Avil\'es. Finally, we will give some partial results about the relation of such topological properties with the combinatorial property of being a $\mathrm{P}$-point.
\end{abstract}

\maketitle

By identifying a subset of $\omega$ with an element of the Cantor set $2^\omega$ in the obvious way (which we will freely do throughout the paper), it is possible to study the topological properties of any $\XX\subseteq\PP(\omega)$. We will focus on the case $\XX=\UU$, where $\UU$ is an ultrafilter on $\omega$. The case $\XX=\FF$, where $\FF$ is simply a filter on $\omega$, has been studied extensively (see Chapter 4 in \cite{bart}). From now on, all filters and ideals are implicitly assumed to be on $\omega$.

First, we will observe that there are many (actually, as many as possible) non-homeomorphic ultrafilters. However, the proof is based on a cardinality argument, hence it is not `honest' in the sense of Van Douwen: it would be desirable to find `quotable' topological properties that distinguish ultrafilters up to homeomorphism. This is consistently achieved in Section 3 using the property of being completely Baire (see Corollary $\ref{ncb}$ and Theorem $\ref{cb}$), in Section 4 using countable dense homogeneity (see Theorem $\ref{ncdh}$ and Theorem $\ref{cdh}$) and in Section 6 using the perfect set property (see Theorem $\ref{noclosedpsp}$ and Corollary $\ref{closedpsp}$).

In Section 5, we will adapt the proof of Theorem $\ref{cdh}$ to obtain the countable dense homogeneity of the $\omega$-power, consistently answering a question of Hru\v{s}\'ak and Zamora Avil\'es from \cite{hrusak} (see Corollary $\ref{consanswer}$).

In Section 7, using a modest large cardinal assumption, we will obtain a strong generalization of the main result of Section 6 (see Theorem $\ref{lrpspu}$).

Finally, in Section 8, we will investigate the relationship between the property of being a $\mathrm{P}$-point and the above topological properties; many questions on this front remain open.

\begin{proposition}\label{many} Let $\UU,\VV\subseteq 2^\omega$ be non-principal ultrafilters. Define $\UU\cong\VV$ if the topological spaces $\UU$ and $\VV$ are homeomorphic. Then the equivalence classes of $\cong$ have size $\cccc$. 
\end{proposition}
\begin{proof}
To show that each equivalence class has size at least $\cccc$, simply use homeomorphisms of $2^\omega$ induced by permutations of $\omega$ and an almost disjoint family of subsets of $\omega$ of size $\cccc$ (see, for example, Lemma 9.21 in \cite{jech}).

By Lavrentiev's lemma (see Theorem 3.9 in \cite{kechris}), if $g:\UU\longrightarrow\VV$ is a homeomorphism, then there exists a homeomorphism $f:G\longrightarrow H$ that extends $g$, where $G$ and $H$ are $G_\delta$ subsets of $2^\omega$. Since there are only $\cccc$ such homeomorphisms, it follows that an equivalence class of $\cong$ has size at most $\cccc$.
\end{proof}

\begin{corollary}
There are $2^\cccc$ pairwise non-homeomorphic non-principal ultrafilters.
\end{corollary}

\section{Notation and Terminology}
Our main reference for descriptive set theory is \cite{kechris}. For other set-theoretic notions, see \cite{bart} or \cite{jech}. For notions that are related to large cardinals, see \cite{kanamori}. For all undefined topological notions, see \cite{engelking}.

By \emph{space} we mean separable metrizable topological space, with a unique exception in Section 6. For every $s\in\ct$, we will denote by $[s]$ the basic clopen set $\{x\in 2^\omega:s\subseteq x\}$.
Given a tree $T\subseteq\ct$, we will denote by $[T]$ the set of branches of $T$, that is $[T]=\{x\in 2^\omega: x\upharpoonright n\in T \textrm{ for all }n\in\omega\}$.

Given a function $f$ and $A\subseteq\dom(f)$, we will denote by $f[A]$ the image of $A$ under $f$, that is $f[A]=\{f(x):x\in A\}$.

A space $X$ is \emph{homogeneous} if whenever $x,y\in X$ there exists a homeomorphism $f:X\longrightarrow X$ such that $f(x)=y$.

Define the homeomorphism $c:2^\omega\longrightarrow 2^\omega$ by setting $c(x)(n)=1-x(n)$ for every $x\in 2^\omega$ and $n\in\omega$. Using $c$, one sees that every ultrafilter $\UU\subseteq 2^\omega$ is homeomorphic to its dual maximal ideal $\JJ=2^\omega\setminus\UU=c[\UU]$.

A \emph{perfect set} in a space $X$ is a non-empty closed subset $P$ of $X$ with no isolated points. Recall that $P$ is a perfect set in $2^\omega$ if and only if it is homeomorphic to $2^\omega$. A \emph{Bernstein set} is a subset $B$ of $X=2^\omega$ such that $B$ and $X\setminus B$ both intersect every perfect set in $X$. Given such a set $B$, since $2^\omega$ is homeomorphic to $2^\omega\times 2^\omega$, one actually has $|P\cap B|=\cccc$ and $|P\cap (X\setminus B)|=\cccc$ for every perfect set $P$ in $X$.

For every $x\subseteq\omega$, define $x^0=\omega\setminus x$ and $x^1=x$. Given a family $\Aa\subseteq\PP(\omega)$, a \emph{word} in $\Aa$ is an intersection of the form
$$
\bigcap_{x\in\tau}x^{w(x)}
$$
for some $\tau\in[\Aa]^{<\omega}$ and $w:\tau\longrightarrow 2$. Recall that $\Aa$ is an \emph{independent family} if every word in $\Aa$ is infinite.

A family $\FF\subseteq\PP(\omega)$ has the \emph{finite intersection property} if $\bigcap\sigma$ is infinite for all $\sigma\in [\FF]^{<\omega}$. Given such a family, we will denote by $\langle\FF\rangle$ the filter generated by $\FF$. Let $\FR$ be the collection of all cofinite subsets of $\omega$. Recall that an ultrafilter $\UU$ is non-principal if and only if $\FR\subseteq\UU$. In particular, every non-principal ultrafilter is dense in $2^\omega$. For any fixed $x\in 2^\omega$, define $x\!\uparrow=\{y\in 2^\omega: x\subseteq y\}$.

Whenever $x,y\in\PP(\omega)$, define $x\subseteq^\ast y$ if $x\setminus y$ is finite. Given $\CC\subseteq\PP(\omega)$, a \emph{pseudointersection} of $\CC$ is a subset $x$ of $\omega$ such that $x\subseteq^\ast y$ for all $y\in\CC$. Given a cardinal $\kappa$, a non-principal ultrafilter $\UU$ is a \emph{$\mathrm{P}_\kappa$-point} if every $\CC\in[\UU]^{<\kappa}$ has a pseudointersection in $\UU$. A \emph{$\mathrm{P}$-point} is simply a $\mathrm{P}_{\omega_1}$-point.

A family $\II\subseteq\PP(\omega)$ has the \emph{finite union property} if $\bigcup\sigma$ is coinfinite for all $\sigma\in [\II]^{<\omega}$. Given such a family, we will denote by $\langle\II\rangle$ the ideal generated by $\II$. Let $\FIN$ be the collection of all finite subsets of $\omega$. For any fixed $x\in 2^\omega$, define $x\!\downarrow=\{y\in 2^\omega: y\subseteq x\}$.

Given $\CC\subseteq\PP(\omega)$, a \emph{pseudounion} of $\CC$ is a subset $x$ of $\omega$ such that $y\subseteq^\ast x$ for all $y\in\CC$. A maximal ideal $\JJ$ is a \emph{$\mathrm{P}$-ideal} if $c[\JJ]$ is a $\mathrm{P}$-point.

\section{Basic properties}

In this section, we will notice that some topological properties are shared by all non-principal ultrafilters. It is easy to realize that every principal ultrafilter $\UU\subseteq 2^\omega$ is homeomorphic to $2^\omega$.

Since any maximal ideal $\JJ$ (actually, any ideal) is a topological subgroup of $2^\omega$ under the operation of symmetric difference (or equivalently, sum modulo $2$), every ultrafilter $\UU=c[\JJ]$ is also a topological group. In particular, every ultrafilter $\UU$ is a homogeneous topological space.

The following proposition is Lemma 3.1 in \cite{fitzpatrick1}.
\begin{proposition}[Fitzpatrick, Zhou]
Let $X$ be a homogeneous topological space. Then $X$ is a Baire space if and only if $X$ is not meager in itself.
\end{proposition}
\begin{proof}
One implication is trivial. Now assume that $X$ is not a Baire space.
Since $X$ is homogeneous, it follows easily that
$$
\BB=\{U:U\textrm{ is a non-empty meager open set in }X\}
$$
is a base for $X$. So $X=\bigcup\BB$ is the union of a collection of meager open sets. Hence $X$ is meager by Banach's category theorem (see Theorem 16.1 in \cite{oxtoby}).

For the convenience of the reader, we sketch the proof in our particular case. Fix a maximal $\CC\subseteq\BB$ consisting of pairwise disjoint sets. Observe that $X\setminus \bigcup\CC$ is closed nowhere dense. For every $U\in\CC$, fix nowhere dense sets $N_n(U)$ such that $U=\bigcup_{n\in\omega}N_n(U)$. It is easy to check that $\bigcup_{U\in\CC}N_n(U)$ is nowhere dense in $X$ for every $n\in\omega$.
\end{proof}

Given any ultrafilter $\UU\subseteq 2^\omega$, notice that $c$ is a homeomorphism of $2^\omega$ such that $2^\omega$ is the disjoint union of $\UU$ and $c[\UU]$. In particular, $\UU$ must be non-meager and non-comeager in $2^\omega$ by Baire's category theorem. Actually, it follows easily from the 0-1 Law that no non-principal ultrafilter $\UU$ can have the property of Baire (see Theorem 8.47 in \cite{kechris}). In particular, no non-principal ultrafilter $\UU$ can be analytic (see Theorem 21.6 in \cite{kechris}) or co-analytic.

\begin{corollary}
Let $\UU\subseteq 2^\omega$ be an ultrafilter. Then $\UU$ is a Baire space.
\end{corollary}
\begin{proof}
If $\UU$ were meager in itself, then it would be meager in $2^\omega$, which is a contradiction.
\end{proof}

On the other hand, by Theorem 8.17 in \cite{kechris}, no non-principal ultrafilter can be a Choquet space (see Section 8.C in \cite{kechris}).

\section{Completely Baire ultrafilters}

\begin{definition}
A space $X$ is \emph{completely Baire} if every closed subspace of $X$ is a Baire space.
\end{definition}

For example, every Polish space is completely Baire. For co-analytic spaces, the converse is also true (see Corollary 21.21 in \cite{kechris}).

In the proof of Theorem $\ref{cb}$, we will need the following characterization (see Corollary 1.9.13 in \cite{vanmill1}). Observe that one implication is trivial.

\begin{lemma}[Hurewicz]\label{hurewicz} A space is completely Baire if and only if it does not contain any closed homeomorphic copy of $\QQQ$. 
\end{lemma}

The following (well-known) lemma is the first step in constructing an ultrafilter that is not completely Baire.

\begin{lemma}\label{perfectindependent} There exists a perfect subset $P$ of $2^\omega$ such that $P$ is an independent family.
\end{lemma}
\begin{proof}
We will give three proofs. The first proof simply shows that the classical construction of an independent family of size $\cccc$ (see for example Lemma 7.7 in \cite{jech}) actually gives a perfect independent family. Define
$$
I=\{(\ell,F):\ell\in\omega,F\subseteq {}^{\ell}2\}.
$$
Since $I$ is a countably infinite set, we can identify $2^I$ and $2^\omega$. The desired independent family will be a collection of subsets of $I$. Consider the function $f:2^\omega\longrightarrow 2^I$ defined by
$$
f(x)=\{(\ell,F):x\upharpoonright\ell\in F\}.
$$
It is easy to check that $f$ is a continuous injection, hence a homeomorphic embedding by compactness. It follows that $P=\ran(f)$ is a perfect set. To check that $P$ is an independent family, fix $\tau\in [P]^{<\omega}$ and $w:\tau\longrightarrow 2$. Suppose that $\tau=f[\sigma]$, where $\sigma=\{ x_1,\ldots x_k\}$ and $x_1,\ldots ,x_k$ are distinct. Choose $\ell$ large enough so that 
$x_1\upharpoonright\ell,\ldots ,x_k\upharpoonright\ell$ are distinct. It follows that
$$
(\ell',\{x\upharpoonright\ell':x\in\sigma\textrm{ and }w(f(x))=1\})\in\bigcap_{y\in\tau}y^{w(y)}
$$
for every $\ell'\geq\ell$, which concludes the proof.

The second proof is also combinatorial. We will inductively construct $k_n\in\omega$ and a finite tree $T_n\subseteq {}^{<\omega}2$ for every $n\in\omega$ so that the following conditions are satisfied.
\begin{enumerate}
\item $k_m<k_n$ whenever $m<n<\omega$.
\item $T_m\subseteq T_n$ whenever $m\leq n<\omega$.
\item All maximal elements of $T_n$ have length $k_n$. We will use the notation $M_n=\{t\in T_n:\dom(t)=k_n\}$.
\item\label{perfect} For every $t\in T_n$ there exist two distinct elements of $T_{n+1}$ whose restriction to $k_n$ is $t$.
\item\label{independent} Given any $v:M_n\longrightarrow 2$, there exists $i\in k_{n+1}\setminus k_n$ such that $t(i)=v(t\upharpoonright k_n)$ for every $t\in M_{n+1}$.
\end{enumerate}
In the end, set $T=\bigcup_{n<\omega}T_n$ and $P=[T]$. Condition $(\ref{perfect})$ guarantees that $P$ is perfect. Next, we will verify that condition $(\ref{independent})$ guarantees that $P$ is an independent family. Fix $\tau\in[P]^{<\omega}$ and $w:\tau\longrightarrow 2$. For all sufficiently large $n\in\omega$, some $v\in{}^{M_n}2$ satisfies $v(x\upharpoonright k_n)=w(x)$ for all $x\in\tau$. By condition $(\ref{independent})$, there exists $i\in k_{n+1}\setminus k_n$ such that
$$
x(i)=(x\upharpoonright k_{n+1})(i)=v(x\upharpoonright k_n)=w(x)
$$
for all $x\in\tau$.

Start with $k_0=0$ and $T_0=\{\varnothing\}$. Given $k_n$ and $T_n$, define $k_{n+1}=k_n+2^{|M_n|}+1$. Fix an enumeration $\{v_j:j\in 2^{|M_n|}\}$ of all functions $v:M_n\longrightarrow 2$. Let $T_{n+1}$ consist of all initial segments of functions $t:k_{n+1}\longrightarrow 2$ such that $t\upharpoonright k_n\in M_n$ and $t(k_n+j)=v_j(t\upharpoonright k_n)$ for all $j<2^{|M_n|}$. Then, condition $(\ref{independent})$ is clearly satisfied. Since there is no restriction on $t(k_n+2^{|M_n|})$, condition $(\ref{perfect})$ is also satisfied. 

The third proof is topological. Fix an enumeration $\{(n_i,w_i):i\in\omega\}$ of all pairs $(n,w)$ such that $n\in\omega$ and $w:n\longrightarrow 2$. Define
$$
R_i=\biggl\{x\in (2^\omega)^{n_i}:\bigcap_{j\in n_i}x_j^{w_i(j)}\textrm{ is infinite}\biggr\}
$$
for every $i\in\omega$ and observe that each $R_i$ is comeager. 
By Exercise 8.8 and Theorem 19.1 in \cite{kechris}, there exists a comeager subset of the Vietoris hyperspace $K(2^\omega)$ consisting  of perfect sets $P\subseteq 2^\omega$ such that $\{x\in P^{n_i}:x_j\neq x_k\textrm{ whenever }j\neq k\}\subseteq R_i$ for every $i\in\omega$. It is trivial to check that any such $P$ is an independent family.
\end{proof}

We remark that, in some sense, the last two proofs that we have given of the above lemma are the same. The Vietoris hyperspace $K(2^\omega)$ is naturally homeomorphic to the space $X$ of pruned subtrees of ${}^{<\omega}2$ with basic open sets of the form $\{T\in X:T\cap {}^{<i}2=\tau\}$ for a fixed pruned subtree $\tau$ of ${}^{<i}2$. Moreover, the set $\{ T\in X:[T]\textrm{ is an independent family}\}$ is comeager in $X$ because the combinatorial proof's rule for constructing $T_{n+1}$ from $T_n$ only needs to be followed infinitely often.

The authors propose to call the following \emph{Kunen's closed embedding trick}.

\begin{theorem}[Kunen]\label{closed} Fix a zero-dimensional space $C$. There exists a non-principal ultrafilter $\UU\subseteq 2^\omega$ that contains a homeomorphic copy of $C$ as a closed subset.
\end{theorem}
\begin{proof}
Fix $P$ as in Lemma $\ref{perfectindependent}$. Since $P$ is homeomorphic to $2^\omega$, we can assume that $C$ is a subspace of $P$. Observe that the family
$$
\GG=C\cup\{\omega\setminus x:x\in P\setminus C\}
$$
has the finite intersection property because $P$ is an independent family. Any non-principal ultrafilter $\UU\supseteq\GG$ will contain $C$ as a closed subset.
\end{proof}
\begin{corollary}\label{ncb}
There exists an ultrafilter $\UU\subseteq 2^\omega$ that is not completely Baire.
\end{corollary}
\begin{proof}
Simply choose $C=\QQQ$.
\end{proof}

Since $2^\omega$ is homeomorphic to $2^\omega\times 2^\omega$, one can easily obtain the following strenghtening of Theorem $\ref{closed}$. Observe that, since any space has at most $\cccc$ closed subsets, the result cannot be improved.
\begin{theorem} Fix a collection $\CC$ of zero-dimensional spaces such that $|\CC|\leq\cccc$. There exists a non-principal ultrafilter $\UU\subseteq 2^\omega$ that contains a homeomorphic copy of $C$ as a closed subset for every $C\in\CC$.
\end{theorem}

The next theorem, together with Corollary $\ref{ncb}$, shows that under $\mathrm{MA(countable)}$ the property of being completely Baire is enough to distinguish ultrafilters up to homeomorphism.

\begin{theorem}\label{cb} Assume that $\mathrm{MA(countable)}$ holds. Then there exists a non-principal ultrafilter $\UU\subseteq 2^\omega$ that is completely Baire.
\end{theorem}
\begin{proof}
Enumerate as $\{Q_\eta:\eta\in\cccc\}$ all subsets of $2^\omega$ that are homeomorphic to $\QQQ$. By Lemma $\ref{hurewicz}$, it will be sufficient to construct a non-principal ultrafilter $\UU$ such that no $Q_\eta$ is a closed subset of $\UU$.

We will construct $\FF_\xi$ for every $\xi\in\cccc$ by transfinite recursion. In the end, let $\UU$ be any ultrafilter extending $\bigcup_{\xi\in\cccc}\FF_\xi$. By induction, we will make sure that the following requirements are satisfied.
\begin{enumerate}
\item $\FF_\mu\subseteq\FF_\eta$ whenever $\mu\leq\eta<\cccc$.
\item $\FF_\xi$ has the finite intersection property for every $\xi\in\cccc$.
\item\label{smallultra} $|\FF_\xi|<\cccc$ for every $\xi\in\cccc$.
\item The potential closed copy of the rationals $Q_\eta$ is dealt with at stage $\xi=\eta+1$: that is, either  $\omega\setminus x\in\FF_\xi$  for some $x\in Q_\eta$ or there exists $x\in\FF_\xi$ such that $x\in\cl(Q_\eta)\setminus Q_\eta$.
\end{enumerate}

Start by letting $\FF_0=\FR$. Take unions at limit stages. At a successor stage $\xi=\eta+1$, assume that $\FF_\eta$ is given. First assume that there exists $x\in Q_\eta$ such that $\FF_\eta\cup\{\omega\setminus x\}$ has the finite intersection property. In this case, simply set $\FF_\xi=\FF_\eta\cup\{\omega\setminus x\}$.

Now assume that $\FF_\eta\cup\{\omega\setminus x\}$ does not have the finite intersection property for any $x\in Q_\eta$. It is easy to check that this implies $Q_\eta\subseteq\langle\FF_\eta\rangle$. Apply Lemma $\ref{keycb}$ with $\FF=\FF_\eta$ and $Q=Q_\eta$ to get $x\in\cl(Q_\eta)\setminus Q_\eta$ such that $\FF_\eta\cup\{x\}$ has the finite intersection property. Finally, set $\FF_\xi=\FF_\eta\cup\{x\}$.
\end{proof}

\begin{lemma}\label{keycb} Assume that $\mathrm{MA(countable)}$ holds. Let $\FF$ be a collection of subsets of $\omega$ with the finite intersection property such that $|\FF|<\cccc$. Let $Q$ be a non-empty subset of $2^\omega$ with no isolated points such that $Q\subseteq\langle\FF\rangle$ and $|Q|<\cccc$. Then there exists $x\in\cl(Q)\setminus Q$ such that $\FF\cup\{x\}$ has the finite intersection property.
\end{lemma}
\begin{proof}
Consider the countable poset 
$$
\PPP=\{s\in\ct:\textrm{there exist }q\in Q\textrm{ and }n\in\omega\textrm{ such that }s=q\upharpoonright n\},
$$
with the natural order given by reverse inclusion.

For every $\sigma=\{x_1,\ldots,x_k\}\in [\FF]^{<\omega}$ and $\ell\in\omega$, define
$$
D_{\sigma,\ell}=\{s\in\PPP:\textrm{there exists }i\in\dom(s)\setminus\ell \textrm{ such that }s(i)=x_1(i)=\cdots=x_k(i)=1\}.
$$
Using the fact that $Q\subseteq\langle\FF\rangle$, it is easy to see that each $D_{\sigma,\ell}$ is dense in $\PPP$.

For every $q\in Q$, define
$$
D_q=\{s\in\PPP:\textrm{there exists }i\in\dom(s)\textrm{ such that }s(i)\neq q(i)\}.
$$
Since $Q$ has no isolated points, each $D_q$ is dense in $\PPP$.

Since $|\FF|<\cccc$ and $|Q|<\cccc$, the collection of dense sets
$$
\DD=\{D_{\sigma,\ell}:\sigma\in [\FF]^{<\omega},\ell\in\omega\}\cup\{D_q:q\in Q\}
$$
has also size less than $\cccc$. Therefore, by $\mathrm{MA(countable)}$, there exists a $\DD$-generic filter $G\subseteq\PPP$. Let $x=\bigcup G\in 2^\omega$. The dense sets of the form $D_{\sigma,\ell}$ ensure that $\FF\cup\{x\}$ has the finite intersection property. The definition of $\PPP$ guarantees that $x\in\cl(Q)$. Finally, the dense sets of the form $D_q$ guarantee that $x\notin Q$.
\end{proof}

\begin{question} Can the assumption that $\mathrm{MA(countable)}$ holds be dropped in Theorem $\ref{cb}$?
\end{question}

\section{Countable dense homogeneity}

\begin{definition} A space $X$ is \emph{countable dense homogeneous} if for every pair $(D,E)$ of countable dense subsets of $X$ there exists a homeomorphism $f:X\longrightarrow X$ such that $f[D]=E$.
\end{definition}

We will start this section by consistently constructing an ultrafilter that is not countable dense homogeneous. We will use Sierpi\'{n}ski's technique for killing homeomorphisms (see \cite{vanmill4} or Appendix 2 of \cite{vandouwen} for a nice introduction). The key lemma is the following.

\begin{lemma}\label{independentncdh} Assume that $\mathrm{MA(countable)}$ holds. Let $\DD$ be a countable independent family that is dense in $2^\omega$. Fix $D_1$ and $D_2$ disjoint countable dense subsets of $\DD$. Then there exists $\Aa\subseteq 2^\omega$ satisfying the following requirements.
\begin{itemize}
\item $\Aa$ is an independent family.
\item $\DD\subseteq\Aa$.
\item If $G\supseteq\DD$ is a $G_\delta$ subset of $2^\omega$ and $f:G\longrightarrow G$ is a homeomorphism such that $f[D_1]=D_2$, then there exists $x\in G$ such that $\{x,\omega\setminus f(x)\}\subseteq\Aa$.
\end{itemize}
\end{lemma}
\begin{proof}
Enumerate as $\{f_\eta:\eta\in\cccc\}$ all homeomorphisms
$$
f_\eta:G_\eta\longrightarrow G_\eta
$$
such that $f_\eta[D_1]=D_2$, where $G_\eta\supseteq\DD$ is a $G_\delta$ subset of $2^\omega$.

We will construct $\Aa_\xi$ for every $\xi\in\cccc$ by transfinite recursion. In the end, set $\Aa=\bigcup_{\xi\in\cccc}\Aa_\xi$. By induction, we will make sure that the following requirements are satisfied.
\begin{enumerate}
\item $\Aa_\mu\subseteq\Aa_\eta$ whenever $\mu\leq\eta<\cccc$.
\item $\Aa_\xi$ is an independent family for every $\xi\in\cccc$.
\item\label{smallsize} $|\Aa_\xi|<\cccc$ for every $\xi\in\cccc$.
\item The homeomorphism $f_\eta$ is dealt with at stage $\xi=\eta+1$: that is, there exists $x\in G_\eta$ such that $\{x,\omega\setminus f_\eta(x)\}\subseteq\Aa_\xi$.
\end{enumerate}

Start by letting $\Aa_0=\DD$. Take unions at limit stages. At a successor stage $\xi=\eta+1$, assume that $\Aa_\eta$ is given.

List as $\{w_\alpha:\alpha\in\kappa\}$ all the words in $\Aa_\eta$, where $\kappa=|\Aa_\eta|<\cccc$ by $(\ref{smallsize})$.
It is easy to check that, for any fixed $n\in\omega$, $\alpha\in\kappa$ and $\varepsilon_1, \varepsilon_2\in 2$, the set
$$
W_{\alpha,n,\varepsilon_1,\varepsilon_2}=\{x\in G_\eta: |w_\alpha\cap x^{\varepsilon_1}\cap f_\eta(x)^{\varepsilon_2}|\geq n\}
$$
is open in $G_\eta$. It is also dense, because $D_1\setminus(F\cup f_\eta^{-1}[F])\subseteq W_{\alpha,n,\varepsilon_1,\varepsilon_2}$, where $F$ consists of the finitely many elements of $\Aa_\eta$ that appear in $w_\alpha$. Therefore, each $W_{\alpha,n,\varepsilon_1,\varepsilon_2}$ is comeager in $2^\omega$. Recall that $\mathrm{MA(countable)}$ is equivalent to $\textrm{cov}(\MM)=\cccc$ (see Theorem 7.13 in \cite{blass} or Theorem 2.4.5 in \cite{bart}). It follows that the intersection
$$
W=\bigcap\{W_{\alpha,n,\varepsilon_1,\varepsilon_2}:\textrm{$n\in\omega$, $\alpha\in\kappa$ and $\varepsilon_1, \varepsilon_2\in 2$}\}
$$
is non-empty. Now simply pick $x\in W$ and set $\Aa_\xi=\Aa_\eta\cup\{x,\omega\setminus f_\eta(x)\}$.
\end{proof}

\begin{theorem}\label{ncdh} Assume that $\mathrm{MA(countable)}$ holds. Then there exists an ultrafilter $\UU\subseteq 2^\omega$ that is not countable dense homogeneous.
\end{theorem}
\begin{proof}
Fix $D_1$, $D_2$ and $\Aa$ as in Lemma $\ref{independentncdh}$. Let $\UU\supseteq\Aa$ be any ultrafilter. Assume, in order to get a contradiction, that $\UU$ is countable dense homogeneous. Let $g:\UU\longrightarrow\UU$ be a homeomorphism such that $g[D_1]=D_2$. By Lavrentiev's lemma, it is possible to extend $g$ to a homeomorphism $f:G\longrightarrow G$, where $G$ is a $G_\delta$ subset of $2^\omega$ (see Exercise 3.10 in \cite{kechris}). By Lemma $\ref{independentncdh}$, there exists $x\in G$ such that $\{x,\omega\setminus f(x)\}\subseteq\Aa\subseteq\UU$, contradicting the fact that $f(x)=g(x)\in\UU$.
\end{proof}

\begin{question} Can the assumption that $\mathrm{MA(countable)}$ holds be dropped in Theorem $\ref{ncdh}$?
\end{question}

When first trying to prove Theorem $\ref{ncdh}$, we attempted to construct a non-principal ultrafilter $\UU$ such that no homeomorphism $g:\UU\longrightarrow\UU$ would be such that $g[\FR]\cap \FR=\varnothing$. This is easily seen to be impossible by choosing $g$ to be the multiplication by any coinfinite $x\in\UU$. Actually, something much stronger holds by the following result of Van Mill (see Proposition 3.4 in \cite{vanmill2}).
\begin{definition}[Van Mill]
A space $X$ has the \emph{separation property} if for every countable subset $A$ of $X$ and every meager subset $B$ of $X$ there exists a homeomorphism $f:X\longrightarrow X$ such that $f[A]\cap B=\varnothing$.
\end{definition}
\begin{proposition}[Van Mill] Let $G$ be a Baire topological group acting on space $X$ that is not meager in itself. Then, for all subsets $A$ and $B$ of $X$ with $A$ countable and $B$ meager, the set of elements $g\in G$ such that $gA\cap B=\varnothing$ is dense in $G$.
\end{proposition}
\begin{corollary}\label{bairegroup} Every Baire topological group has the separation property.
\end{corollary}
\begin{corollary} Every ultrafilter $\UU\subseteq 2^\omega$ has the separation property.
\end{corollary}

It is easy to see that, for Baire spaces, being countable dense homogeneous is stronger than having the separation property. On the other hand, the product of $2^\omega$ and the one-dimensional sphere $S^1$ is a compact topological group that has the separation property but is not countable dense homogeneous (see Corollary 3.6 and Remark 3.7 in \cite{vanmill2}). Theorem $\ref{ncdh}$ consistently gives a zero-dimensional topological group with the same feature. Notice that such an example cannot be compact (or even Polish) by the following paragraph.

Recall that a space $X$ is \emph{strongly locally homogeneous} if it admits an open base $\BB$ such that whenever $U\in\BB$ and $x,y\in U$ there exists a homeomorphism $f:X\longrightarrow X$ such that $f(x)=y$ and $f\upharpoonright X\setminus U$ is the identity. For example, any homogeneous zero-dimensional space is strongly locally homogeneous. For Polish spaces, strong local homogeneity implies countable dense homogeneity (see Theorem 5.2 in \cite{anderson}). In \cite{vanmill3}, Van Mill constructed a homogeneous Baire space that is strongly locally homogeneous but not countable dense homogeneous. Actually, his example does not even have the separation property (see Theorem 3.5 in \cite{vanmill3}), so it cannot be a topological group by Corollary $\ref{bairegroup}$. In this sense, our example from Theorem $\ref{ncdh}$ is better than his. On the other hand, his example is constructed in ZFC, while ours needs $\mathrm{MA(countable)}$. Furthermore, his example can be easily modified to have any given dimension (see Remark 4.1 in \cite{vanmill3}).

Next, we will construct (still under $\mathrm{MA(countable)}$) a non-principal ultrafilter that is countable dense homogeneous. In \cite{baldwin}, Baldwin and Beaudoin used $\mathrm{MA(countable)}$ to construct a homogeneous Bernstein subset of $2^\omega$ that is countable dense homogeneous. Both examples give a consistent answer to Question 389 in \cite{fitzpatrick2}, which asks whether there exists a countable dense homogeneous space that is not completely metrizable. In \cite{farah}, using metamathematical methods, Farah, Hru\v{s}\'ak and Mart\'inez Ranero showed that the answer to such question is `yes' in ZFC.

The following lemma will be one of the key ingredients. The other key ingredient is the poset used in the proof of Lemma $\ref{keycdh}$, which was inspired by the poset used in the proof of Lemma 3.1 in \cite{baldwin}.

\begin{lemma}\label{davelemma}
Let $f:2^\omega\longrightarrow 2^\omega$ be a homeomorphism. Fix a non-principal maximal ideal $\JJ\subseteq 2^\omega$ and a countable dense subset $D$ of $\JJ$. Then $f$ restricts to a homeomorphism of $\JJ$ if and only if $\cl(\{d+f(d):d\in D\})\subseteq\JJ$.
\end{lemma}
\begin{proof}
Assume that $f$ restricts to a homeomorphism of $\JJ$. It is easy to check that the function $g:2^\omega\longrightarrow 2^\omega$ defined by $g(x)=x+f(x)$ has range contained in $\JJ$. Since $g$ is continuous, its range must be compact, hence closed in $2^\omega$.

Now assume that $\cl(\{d+f(d):d\in D\})\subseteq\JJ$. Let $x\in 2^\omega$. Fix $d_n\in D$ for $n\in\omega$ so that $\lim_{n\to\infty}d_n=x$. By continuity,
$$
x+f(x)=\lim_{n\to\infty}\left(d_n+f(d_n)\right)\in\JJ.
$$
The proof is concluded by observing that if $a,b\in 2^\omega$ are such that $a+b\in\JJ$, then either $\{a,b\}\subseteq\JJ$ or $\{a,b\}\subseteq 2^\omega\setminus\JJ$.
\end{proof}

\begin{theorem}\label{cdh} Assume that $\mathrm{MA(countable)}$ holds. Then there exists a non-principal ultrafilter $\UU\subseteq 2^\omega$ that is countable dense homogeneous.
\end{theorem}
\begin{proof}
For notational convenience, we will construct a maximal ideal $\JJ\subseteq 2^\omega$ containing all finite sets that is countable dense homogeneous. Enumerate as $\{(D_\eta,E_\eta):\eta\in\cccc\}$ all pairs of countable dense subsets of $2^\omega$.

We will construct $\II_\xi$ for every $\xi\in\cccc$ by transfinite recursion. In the end, let $\JJ$ be any maximal ideal extending $\bigcup_{\xi\in\cccc}\II_\xi$. By induction, we will make sure that the following requirements are satisfied.
\begin{enumerate}
\item $\II_\mu\subseteq\II_\eta$ whenever $\mu\leq\eta<\cccc$.
\item $\II_\xi$ has the finite union property for every $\xi\in\cccc$.
\item\label{smallideal} $|\II_\xi|<\cccc$ for every $\xi\in\cccc$.
\item\label{appdavelemma} The pair $(D_\eta,E_\eta)$ is dealt with at stage $\xi=\eta+1$: that is, either $\omega\setminus x\in\II_\xi$ for some $x\in D_\eta\cup E_\eta$ or there exists $x\in\II_\xi$ and a homeomorphism $f_\eta:2^\omega\longrightarrow 2^\omega$ such that $f_\eta[D_\eta]=E_\eta$ and $\{d+f_\eta(d):d\in D_\eta\}\subseteq x\!\downarrow$.
\end{enumerate}
Observe that, by Lemma $\ref{davelemma}$, the second part of condition $(\ref{appdavelemma})$ guarantees that any maximal ideal $\JJ$ extending $\II_\xi$ will be such that $f_\eta:2^\omega\longrightarrow 2^\omega$ restricts to a homeomorphism of $\JJ$.

Start by letting $\II_0=\FIN$. Take unions at limit stages. At a successor stage $\xi=\eta+1$, assume that $\II_\eta$ is given. First assume that there exists $x\in D_\eta\cup E_\eta$ such that $\II_\eta\cup\{\omega\setminus x\}$ has the finite union property. In this case, we can just set $\II_\xi=\II_\eta\cup\{\omega\setminus x\}$.

Now assume that $\II_\eta\cup\{\omega\setminus x\}$ does not have the finite union property for any $x\in D_\eta\cup E_\eta$. It is easy to check that this implies $D_\eta\cup E_\eta\subseteq\langle\II_\eta\rangle$. Let $x$ and $f$ be given by applying Lemma $\ref{keycdh}$ with $\II=\II_\eta$, $D=D_\eta$ and $E=E_\eta$. Finally, set $\II_\xi=\II_\eta\cup\{x\}$ and $f_\eta=f$.
\end{proof}

\begin{lemma}\label{keycdh} Assume that $\mathrm{MA(countable)}$ holds. Let $\II\subseteq 2^\omega$ be a collection of subsets of $\omega$ with the finite union property and assume that $|\II|<\cccc$. Fix two countable dense subsets $D$ and $E$ of $2^\omega$ such that $D\cup E\subseteq\langle\II\rangle$. Then there exists a homeomorphism $f:2^\omega\longrightarrow2^\omega$ and $x\in 2^\omega$ such that $f[D]=E$, $\II\cup\{x\}$ still has the finite union property and $\{d+f(d):d\in D\}\subseteq x\!\downarrow$.
\end{lemma}
\begin{proof}
Consider the countable poset $\PPP$ consisting of all triples of the form $p=(s,g,\pi)=(s_p,g_p,\pi_p)$ such that, for some $n=n_p\in\omega$, the following requirements are satisfied.
\begin{itemize}
\item $s:n\longrightarrow 2$.
\item $g$ is a bijection between a finite subset of $D$ and a finite subset of $E$.
\item $\pi$ is a permutation of ${}^{n}2$.
\end{itemize}
Furthermore, we require the following compatibility conditions to be satisfied. Condition $(\ref{pis})$ will actually ensure that $\{d+f(d):d\in 2^\omega\}\subseteq x\!\downarrow$. Notice that this is equivalent to $(d+f(d))(i)\leq x(i)$ for all $d\in 2^\omega$ and $i\in\omega$.
\begin{enumerate}
\item\label{pis} $(t+\pi(t))(i)=1$ implies $s(i)=1$ for every $t\in{}^{n}2$ and $i\in n$.
\item\label{pig} $\pi(d\upharpoonright n)=g(d)\upharpoonright n$ for every $d\in\dom(g)$.
\end{enumerate}
Order $\PPP$ by declaring $q\leq p$ if the following conditions are satisfied.
\begin{itemize}
\item $s_q\supseteq s_p$.
\item $g_q\supseteq g_p$.
\item $\pi_q(t)\upharpoonright n_p=\pi_p(t\upharpoonright n_p)$ for all $t\in {}^{n_q}2$.
\end{itemize}

For each $d\in D$, define
$$
D^{\dom}_d=\{p\in\PPP:d\in\dom(g_p)\}.
$$
Given $p\in\PPP$ and $d\in D\setminus \dom(g_p)$, one can simply choose $e\in E\setminus \ran(g_p)$ such that $e\upharpoonright n_p=\pi_p(d\upharpoonright n_p)$. This choice will make sure that $q=(s_p,g_p\cup\{(d,e)\},\pi_p)\in\PPP$. Furthermore it is clear that $q\leq p$. So each $D^{\dom}_d$ is dense in $\PPP$. 

For each $e\in E$, define
$$
D^{\ran}_e=\{p\in\PPP:e\in\ran(g_p)\}.
$$
As above, one can easily show that each $D^{\ran}_e$ is dense in $\PPP$.

For every $\sigma=\{x_1,\ldots, x_k\}\in [\II]^{<\omega}$ and $\ell\in\omega$, define
$$
D_{\sigma,\ell}=\{p\in\PPP:\textrm{there exists }i\in n_p\setminus\ell \textrm{ such that }s_p(i)=x_1(i)=\cdots=x_k(i)=0\}.
$$
Next, we will prove that each $D_{\sigma,\ell}$ is dense in $\PPP$. So fix $\sigma$ and $\ell$ as above. Let $p=(s,g,\pi)\in\PPP$ with $n_p=n$. Find $n'\geq\ell,n$ such that the following conditions hold.
\begin{itemize}
\item All $d\upharpoonright n'$ for $d\in\dom(g)$ are distinct.
\item All $e\upharpoonright n'$ for $e\in\ran(g)$ are distinct.
\item $x_1(n')=\cdots=x_k(n')=d(n')=e(n')=0$ for all $d\in\dom(g)$, $e\in\ran(g)$.
\end{itemize}
This is possible because $\II$ has the finite union property and
$$
\sigma\cup\dom(g)\cup\ran(g)\subseteq\langle\II\rangle.
$$
We can choose a permutation $\pi'$ of ${}^{n'}2$ such that $\pi'(d\upharpoonright n')=g(d)\upharpoonright n'$ for every $d\in\dom(g)$ and $\pi'(t)\upharpoonright n=\pi(t\upharpoonright n)$ for all $t\in {}^{n'}2$. Extend $s$ to $s':n'\longrightarrow 2$ by setting $s'(i)=1$ for every $i\in [n,n')$. It is clear that $p'=(s',g,\pi')\in\PPP$ and $p'\leq p$.

Now let $\pi''$ be the permutation of ${}^{n'+1}2$ obtained by setting
$$
\pi''(t)=\pi'(t\upharpoonright n')^\frown t(n')
$$
for all $t\in {}^{n'+1}2$. Extend $s'$ to $s'':n'+1\longrightarrow 2$ by setting $s''(n')=0$. It is easy to check that $p''=(s'',g,\pi'')\in D_{\sigma,\ell}$ and $p''\leq p'$.

Since $|\II|<\cccc$, the collection of dense sets
$$
\DD=\{D_{\sigma,\ell}:\sigma\in [\II]^{<\omega},\ell\in\omega\}\cup\{D^\dom_d:d\in D\}\cup\{D^\ran_e:e\in E\}
$$
has also size less than $\cccc$. Therefore, by $\mathrm{MA(countable)}$, there exists a $\DD$-generic filter $G\subseteq\PPP$. Define $x=\bigcup\{s_p:p\in G\}$. To define $f(y)(i)$, for a given $y\in 2^\omega$ and $i\in\omega$, choose any $p\in G$ such that $i\in n_p$ and set $f(y)(i)=\pi_p(y\upharpoonright n_p)(i)$.
\end{proof}

\begin{question} Can the assumption that $\mathrm{MA(countable)}$ holds be dropped in Theorem $\ref{cdh}$?
\end{question}

By Theorem 2.3 in \cite{hrusak}, every analytic countable dense homogeneous space must be completely Baire. So the following question seems natural. See also Theorem 2.6 in \cite{hrusak}.

\begin{question} Is a countable dense homogeneous ultrafilter $\UU\subseteq 2^\omega$ necessarily completely Baire?
\end{question}

\section{A question of Hru\v{s}\'ak and Zamora Avil\'es}
The main result of \cite{hrusak} states that, given a Borel subset $X$ of $2^\omega$, the following statements are equivalent.
\begin{itemize}
\item $X^\omega$ is countable dense homogeneous.
\item $X$ is a $G_\delta$.
\end{itemize}
Question 3.2 in the same paper asks whether there exists a non-$G_\delta$ subset $X$ of $2^\omega$ such that $X^\omega$ is countable dense homogeneous. By a rather straightforward modification of the proof of Theorem $\ref{cdh}$, we will give a consistent answer to such question (see Corollary $\ref{consanswer}$).

Our example is also relevant to the second half of Question 387 in \cite{fitzpatrick2}, which asks to characterize the zero-dimensional spaces $X$ such that $X^\omega$ is countable dense homogeneous. 

Observe that, given any ideal $\II\subseteq 2^\omega$, the infinite product $\II^\omega$ inherits the structure of topological group using coordinate-wise addition. The following lemma is proved exactly like the corresponding half of Lemma $\ref{davelemma}$.
\begin{lemma}\label{davelemmapower}
Let $f:(2^\omega)^\omega\longrightarrow (2^\omega)^\omega$ be a homeomorphism. Fix a non-principal maximal ideal $\JJ\subseteq 2^\omega$ and a countable dense subset $D$ of $\JJ^\omega$. If $\cl(\{d+f(d):d\in D\})\subseteq\JJ^\omega$ then $f$ restricts to a homeomorphism of $\JJ^\omega$.
\end{lemma}

\begin{theorem}\label{cdhpower} Assume that $\mathrm{MA(countable)}$ holds. Then there exists a non-principal ultrafilter $\UU\subseteq 2^\omega$ such that $\UU^\omega$ is countable dense homogeneous.
\end{theorem}
\begin{proof}
For notational convenience, we will construct a maximal ideal $\JJ\subseteq 2^\omega$ containing all finite sets such that $\JJ^\omega$ is countable dense homogeneous. Enumerate as $\{(D_\eta,E_\eta):\eta\in\cccc\}$ all pairs of countable dense subsets of $(2^\omega)^\omega$.

We will construct $\II_\xi$ for every $\xi\in\cccc$ by transfinite recursion. In the end, let $\JJ$ be any maximal ideal extending $\bigcup_{\xi\in\cccc}\II_\xi$. By induction, we will make sure that the following requirements are satisfied. Let $P_\eta=\bigcup_{i\in\omega}\pi_i[D_\eta\cup E_\eta]$, where $\pi_i:(2^\omega)^\omega\longrightarrow 2^\omega$ is the natural projection.
\begin{enumerate}
\item $\II_\mu\subseteq\II_\eta$ whenever $\mu\leq\eta<\cccc$.
\item $\II_\xi$ has the finite union property for every $\xi\in\cccc$.
\item\label{smallidealpower} $|\II_\xi|<\cccc$ for every $\xi\in\cccc$.
\item\label{appdavelemmapower} The pair $(D_\eta,E_\eta)$ is dealt with at stage $\xi=\eta+1$: that is, either $\omega\setminus x\in\II_\xi$ for some $x\in P_\eta$ or there exists $x_i\in\II_\xi$ for every $i\in\omega$ and a homeomorphism $f_\eta:(2^\omega)^\omega\longrightarrow (2^\omega)^\omega$ such that $f_\eta[D_\eta]=E_\eta$ and $\{d+f_\eta(d):d\in D_\eta\}\subseteq\prod_{i\in\omega}(x_i\!\downarrow)$.
\end{enumerate}
Observe that, by Lemma $\ref{davelemmapower}$, the second part of condition $(\ref{appdavelemma})$ guarantees that any maximal ideal $\JJ$ extending $\II_\xi$ will be such that $f_\eta:(2^\omega)^\omega\longrightarrow (2^\omega)^\omega$ restricts to a homeomorphism of $\JJ^\omega$.

Start by letting $\II_0=\FIN$. Take unions at limit stages. At a successor stage $\xi=\eta+1$, assume that $\II_\eta$ is given. First assume that there exists $x\in P_\eta$ such that $\II_\eta\cup\{\omega\setminus x\}$ has the finite union property. In this case, we can just set $\II_\xi=\II_\eta\cup\{\omega\setminus x\}$.

Now assume that $\II_\eta\cup\{\omega\setminus x\}$ does not have the finite intersection property for any $x\in P_\eta$. It is easy to check that this implies $P_\eta\subseteq\langle\II_\eta\rangle$, hence $D_\eta\cup E_\eta\subseteq\langle\II_\eta\rangle^\omega$. Let $x_i$ for $i\in\omega$ and $f$ be given by applying Lemma $\ref{keycdhpower}$ with $\II=\II_\eta$, $D=D_\eta$ and $E=E_\eta$. Finally, set $\II_\xi=\II_\eta\cup\{x_i:i\in\omega\}$ and $f_\eta=f$.
\end{proof}

\begin{lemma}\label{keycdhpower} Assume that $\mathrm{MA(countable)}$ holds. Let $\II\subseteq 2^\omega$ be a collection of subsets of $\omega$ with the finite union property and assume that $|\II|<\cccc$. Fix two countable dense subsets $D$ and $E$ of $(2^\omega)^\omega$ such that $D\cup E\subseteq\langle\II\rangle^\omega$. Then there exists a homeomorphism $f:(2^\omega)^\omega\longrightarrow (2^\omega)^\omega$ and $x_i\in 2^\omega$ for $i\in\omega$ such that $f[D]=E$, $\II\cup\{x_i:i\in\omega\}$ still has the finite union property and $\{d+f(d):d\in D\}\subseteq\prod_{i\in\omega}(x_i\!\downarrow)$.
\end{lemma}

\begin{proof}
We will make a natural identification of $(2^\omega)^\omega$ with $2^{\omega\times\omega}$. Namely, we will identify a sequence $(x_i)_{i\in\omega}$ with the function $x$ given by $x(i,j)=x_i(j)$.

Consider the countable poset $\PPP$ consisting of all triples of the form $p=(s,g,\pi)=(s_p,g_p,\pi_p)$ such that, for some $m=m_p\in\omega$ and $n=n_p\in\omega$, the following requirements are satisfied.
\begin{itemize}
\item $s:m\times n\longrightarrow 2$.
\item $g$ is a bijection between a finite subset of $D$ and a finite subset of $E$.
\item $\pi$ is a permutation of ${}^{m\times n}2$.
\end{itemize}
Furthermore, we require the following compatibility conditions to be satisfied. Condition $(\ref{pispower})$ will actually ensure that $\{d+f(d):d\in (2^\omega)^\omega\}\subseteq \prod_{i\in\omega}(x_i\!\downarrow)$. Notice that this is equivalent to $(d+f(d))(i,j)\leq x(i,j)=x_i(j)$ for all $d\in 2^{\omega\times\omega}$ and $(i,j)\in\omega\times\omega$.
\begin{enumerate}
\item\label{pispower} $(t+\pi(t))(i,j)=1$ implies $s(i,j)=1$ for every $t\in{}^{m\times n}2$ and $(i,j)\in m\times n$.
\item\label{pigpower} $\pi(d\upharpoonright (m\times n))=g(d)\upharpoonright (m\times n)$ for every $d\in\dom(g)$.
\end{enumerate}
Order $\PPP$ by declaring $q\leq p$ if the following conditions are satisfied.
\begin{itemize}
\item $s_q\supseteq s_p$.
\item $g_q\supseteq g_p$.
\item $\pi_q(t)\upharpoonright (m_p\times n_p)=\pi_p(t\upharpoonright (m_p\times n_p))$ for all $t\in {}^{m_q\times n_q}2$.
\end{itemize}

For each $d\in D$, define
$$
D^{\dom}_d=\{p\in\PPP:d\in\dom(g_p)\}.
$$
Given $p\in\PPP$ and $d\in D\setminus \dom(g_p)$, one can simply choose $e\in E\setminus \ran(g_p)$ such that $e\upharpoonright (m_p\times n_p)=\pi_p(d\upharpoonright (m_p\times n_p))$. This choice will make sure that $q=(s_p,g_p\cup\{(d,e)\},\pi_p)\in\PPP$. Furthermore it is clear that $q\leq p$. So each $D^{\dom}_d$ is dense in $\PPP$. 

For each $e\in E$, define
$$
D^{\ran}_e=\{p\in\PPP:e\in\ran(g_p)\}.
$$
As above, one can easily show that each $D^{\ran}_e$ is dense in $\PPP$.

For each $\sigma=\{x_1,\ldots, x_k\}\in [\II]^{<\omega}$ and $\ell\in\omega$, define
$$
D_{\sigma,\ell}=\{p\in\PPP:\textrm{there exists }j\in n_p\setminus\ell\textrm{ such that }
$$
$$
s_p(0,j)=\cdots=s_p(m_p-1,j)=x_1(j)=\cdots=x_k(j)=0\}.
$$
Next, we will prove that each $D_{\sigma,\ell}$ is dense in $\PPP$. So fix $\sigma$ and $\ell$ as above. Let $p=(s,g,\pi)\in\PPP$ with $m_p=m$ and $n_p=n$. Find $m'\geq m$ and $n'\geq\ell,n$ such that the following conditions hold.
\begin{itemize}
\item All $d\upharpoonright (m'\times n')$ for $d\in\dom(g)$ are distinct.
\item All $e\upharpoonright (m'\times n')$ for $e\in\ran(g)$ are distinct.
\item $x_1(n')=\cdots=x_k(n')=d_i(n')=e_i(n')=0$ for all $d\in\dom(g)$, $e\in\ran(g)$ and $i\in m'$.
\end{itemize}
This is possible because $\II$ has the finite union property and 
$$
\sigma\cup\{d_i:d\in\dom(g),i\in\omega\}\cup\{e_i:e\in\ran(g),i\in\omega\}\subseteq\langle\II\rangle.
$$
We can choose a permutation $\pi'$ of ${}^{m'\times n'}2$ such that $\pi'(d\upharpoonright (m'\times n'))=g(d)\upharpoonright (m'\times n')$ for every $d\in\dom(g)$ and $\pi'(t)\upharpoonright(m\times n)=\pi(t\upharpoonright (m\times n))$ for all $t\in {}^{m'\times n'}2$. Extend $s$ to $s':m'\times n'\longrightarrow 2$ by setting $s'(i,j)=1$ for every $(i,j)\in (m'\times n')\setminus (m\times n)$. It is clear that $p'=(s',g,\pi')\in\PPP$ and $p'\leq p$.

Now let $\pi''$ be the permutation of ${}^{m'\times (n'+1)}2$ obtained by setting
$$
\pi''(t)(i,j)=
\begin{cases} \pi'(t\upharpoonright (m'\times n'))(i,j) &\textrm{if }(i,j)\in m'\times n'\\
t(i,j) &\textrm{if }(i,j)\in m'\times \{n'\}
\end{cases}
$$
for all $t\in {}^{m'\times (n'+1)}2$. Extend $s'$ to $s'':m'\times (n'+1)\longrightarrow 2$ by setting $s''(i,j)=0$ for all $(i,j)\in m'\times \{n'\}$. It is easy to check that $p''=(s'',g,\pi'')\in D_{\sigma,\ell}$ and $p''\leq p'$.

We will need one last class of dense sets. For any given $\ell\in\omega$, define
$$
D_\ell=\{p\in\PPP:m_p\geq\ell\}.
$$
An easier version of the above argument shows that each $D_\ell$ is in fact dense.

Since $|\II|<\cccc$, the collection of dense sets
$$
\DD=\{D_{\sigma,\ell}:\sigma\in [\II]^{<\omega},\ell\in\omega\}\cup\{D^\dom_d:d\in D\}\cup\{D^\ran_e:e\in E\}\cup\{D_\ell:\ell\in\omega\}
$$
has also size less than $\cccc$. Therefore, by $\mathrm{MA(countable)}$, there exists a $\DD$-generic filter $G\subseteq\PPP$. Define $x_i=\bigcup\{s_p(i,-):p\in G\}$ for every $i\in\omega$. To define $f(y)(i,j)$, for a given $y\in 2^{\omega\times\omega}$ and $(i,j)\in\omega\times\omega$, choose any $p\in G$ such that $(i,j)\in m_p\times n_p$ and set $f(y)(i,j)=\pi_p(y\upharpoonright (m_p\times n_p))(i,j)$.
\end{proof}

\begin{corollary}\label{consanswer} Assume that $\mathrm{MA(countable)}$ holds. Then there exists a non-$G_\delta$ subset $X$ of $2^\omega$ such that $X^\omega$ is countable dense homogeneous. 
\end{corollary}

\begin{question} Can the assumption that $\mathrm{MA(countable)}$ holds be dropped in Theorem $\ref{cdhpower}$?
\end{question}

\begin{question} Is there an analytic non-$G_\delta$ subset $X$ of $2^\omega$ such that $X^\omega$ is countable dense homogeneous? Co-analytic?
\end{question}

\section{The perfect set property}

\begin{definition}
Let X be a space. We will say that $A\subseteq X$ has the \emph{perfect set property} if $A$ is either countable or it contains a perfect set.
\end{definition}

It is a classical result of descriptive set theory, due to Souslin, that every analytic subset of a Polish space has the perfect set property (see, for example, Theorem 29.1 in \cite{kechris}).

The following is an easy application of Kunen's closed embedding trick.
\begin{theorem}\label{noclosedpsp} There exists an ultrafilter $\UU\subseteq 2^\omega$ with a closed subset of cardinality $\cccc$ that does not have the perfect set property.
\end{theorem}
\begin{proof}
Fix a Bernstein set $B$ in $2^\omega$, then apply Theorem $\ref{closed}$ with $C=B$.
\end{proof}

Next, we will consistently construct a non-principal ultrafilter $\UU$ such that every closed subset of $\UU$ has the perfect set property. Actually, we will get a much stronger result (see Theorem $\ref{analyticpsp}$).

Recall that a play of the \emph{strong Choquet game} on a topological space $(X,\TT)$ is of the form
\begin{center}
\begin{tabular}{cccccl}
I & $(q_0,U_0)$ & & $(q_1,U_1)$ & & $\cdots$  \\ \hline
II & & $V_0$ & & $V_1$ & $\cdots$,
\end{tabular}
\end{center}
where $U_n,V_n\in\TT$ are such that $q_n\in V_n\subseteq U_n$ and $U_{n+1}\subseteq V_n$ for every $n\in\omega$. Player II wins if $\bigcap_{n\in\omega} U_n\neq\varnothing$. The topological space $(X,\TT)$ is \emph{strong Choquet} if II has a winning strategy in the above game. See Section 8.D in \cite{kechris}.

Define an \emph{$\mathrm{A}$-triple} to be a triple of the form $(\TT,A,Q)$ such that the following conditions are satisfied.
\begin{itemize}
\item $\TT$ is a strong Choquet, second-countable topology on $2^\omega$ that is finer than the standard topology.
\item $A\in\TT$.
\item $Q$ is a non-empty countable subset of $A$ with no isolated points in the subspace topology it inherits from $\TT$.
\end{itemize}
By Theorem 25.18 in \cite{kechris}, for every analytic $A$ there exists a topology $\TT$ as above. Also, by Exercise 25.19 in \cite{kechris}, such a topology $\TT$ necessarily consists only of analytic sets. In particular, all $\mathrm{A}$-triples can be enumerated in type $\cccc$.

\begin{theorem}\label{analyticpsp} Assume that $\mathrm{MA(countable)}$ holds. Then there exists a non-principal ultrafilter $\UU\subseteq 2^\omega$ such that $A\cap\UU$ has the perfect set property for every analytic $A\subseteq 2^\omega$.
\end{theorem}
\begin{proof}
Enumerate as $\{(\TT_\eta, A_\eta, Q_\eta):\eta\in\cccc\}$ all $\mathrm{A}$-triples, making sure that each triple appears cofinally often. Also, enumerate as $\{z_\eta:\eta\in\cccc\}$ all subsets of $\omega$.

We will construct $\FF_\xi$ for every $\xi\in\cccc$ by transfinite recursion. By induction, we will make sure that the following requirements are satisfied.
\begin{enumerate}
\item $\FF_\mu\subseteq\FF_\eta$ whenever $\mu\leq\eta<\cccc$.
\item\label{fippsp} $\FF_\xi$ has the finite intersection property for every $\xi\in\cccc$.
\item $|\FF_\xi|<\cccc$ for every $\xi\in\cccc$.
\item\label{ultra} By stage $\xi=\eta+1$, we must have decided whether $z_\eta\in\UU$: that is, $z_\eta^\varepsilon\in\FF_\xi$ for some $\varepsilon\in 2$.
\item\label{perfectsubset} If $Q_\eta\subseteq\FF_\eta$ then, at stage $\xi=\eta+1$, we will deal with $A_\eta$: that is, there exists $x\in\FF_\xi$ such that $x\!\uparrow\cap A_\eta$ contains a perfect subset.
\end{enumerate}
In the end, let $\UU=\bigcup_{\xi\in\cccc}\FF_\xi$. Notice that $\UU$ will be an ultrafilter by ($\ref{ultra}$).

Start by letting $\FF_0=\FR$. Take unions at limit stages. At a successor stage $\xi=\eta+1$, assume that $\FF_\eta$ is given. First assume that $Q_\eta\nsubseteq\FF_\eta$. In this case, simply set $\FF_\xi=\FF_\eta\cup\{z_\eta^\varepsilon\}$ for a choice of $\varepsilon\in 2$ that is compatible with condition $(\ref{fippsp})$.

Now assume that $Q_\eta\subseteq\FF_\eta$. Apply Lemma $\ref{keypsp}$ with $\FF=\FF_\eta$, $A=A_\eta$, $Q=Q_\eta$ and $\TT=\TT_\eta$ to get a perfect set $P\subseteq A$ such that $\FF_\eta\cup\{\bigcap P\}$ has the finite intersection property. Let $x=\bigcap P$. Set $\FF_\xi=\FF_\eta\cup\{x,z_\eta^\varepsilon\}$, for some $\varepsilon\in 2$ compatible with condition $(\ref{fippsp})$.

Finally, we will check that $\UU$ has the required property. Assume that $A$ is an analytic subset of $2^\omega$ such that $A\cap\UU$ is uncountable. By Theorem 25.18 in \cite{kechris}, there exists a second-countable, strong Choquet topology $\TT$ on $2^\omega$ that is finer than the standard topology and contains $A$. Since every second countable, uncountable Hausdorff space contains a non-empty countable subspace with no isolated points, we can find such a subspace $Q\subseteq A\cap\UU$. Since $\cf(\cccc)>\omega$, there exists $\mu\in\cccc$ such that $Q\subseteq\FF_\mu$. Since we listed each $\mathrm{A}$-triple cofinally often, there exists $\eta\geq\mu$ such that $(\TT,A,Q)=(\TT_\eta,A_\eta,Q_\eta)$. Condition ($\ref{perfectsubset}$) guarantees that $\UU\cap A$ will contain a perfect subset.
\end{proof}

\begin{lemma}\label{keypsp} Assume that $\mathrm{MA(countable)}$ holds. Let $\FF$ be a collection of subsets of $\omega$ with the finite intersection property such that $|\FF|<\cccc$. Suppose that $(\TT,A,Q)$ is an $\mathrm{A}$-triple with $Q\subseteq\FF$. Then there exists a perfect subset $P$ of $A$ such that $\FF\cup\{\bigcap P\}$ has the finite intersection property.
\end{lemma}
\begin{proof}
Fix a winning strategy $\Sigma$ for player II in the strong Choquet game in $(2^\omega,\TT)$. Also, fix a countable base $\BB$ for $(2^\omega,\TT)$. Let $\PPP$ be the countable poset consisting of all functions $p$ such that, for some $n=n_p\in\omega$, the following conditions hold.
\begin{enumerate}
\item $p:{}^{\leq n}2\longrightarrow Q\times\BB$. We will use the notation $p(s)=(q_s^p,U_s^p)$.
\item $U_\varnothing^p=A$.
\item\label{disjointness} For every $s,t\in{}^{\leq n}2$, if $s$ and $t$ are incompatible (that is, $s\nsubseteq t$ and $t\nsubseteq s$) then $U^p_s\cap U^p_t=\varnothing$.
\item For every $s\in{}^{n}2$,
\begin{center}
\begin{tabular}{cccccccc}
I & $(q^p_{s\upharpoonright 0},U^p_{s\upharpoonright 0})$ & & $(q^p_{s\upharpoonright 1},U^p_{s\upharpoonright 1})$ & & $\cdots$ & $(q^p_{s\upharpoonright n},U^p_{s\upharpoonright n})$ & \\ \hline
II & & $V^p_{s\upharpoonright 0}$ & & $V^p_{s\upharpoonright 1}$ & $\cdots$ & & $V^p_{s\upharpoonright n}$
\end{tabular}
\end{center}
is a partial play of the strong Choquet game in $(2^\omega,\TT)$, where the open sets $V^p_{s\upharpoonright i}$ played by II are the ones dictated by the strategy $\Sigma$.
\end{enumerate}
Order $\PPP$ by setting $p\leq p'$ whenever $p\supseteq p'$.

For every $\ell\in\omega$, define
$$
D_\ell=\{p\in\PPP:n_p\geq\ell\}.
$$
Since $Q$ has no isolated points and $\TT$ is Hausdorff, it is easy to see that each $D_\ell$ is dense.

For any fixed $\ell\in\omega$, consider the partition of $2^\omega$ in clopen sets $\PP_\ell=\{[s]:s\in{}^{\ell}2\}$, then define
$$
D^\textrm{ref}_\ell=\{p\in\PPP: \{U_s^p:s\in {}^{n_p}2\}\textrm{ refines }\PP_\ell\}.
$$
Let us check that each $D^\textrm{ref}_\ell$ is dense. Given $p\in\PPP$ and $\ell\in\omega$, let $n=n_p$ and $q^0_s=q^p_s$ for every $s\in{}^{n}2$. Since $Q$ has no isolated points, it is possible, for every $s\in{}^{n}2$, to choose $q^1_s\neq q^0_s$ such that $q^1_s\in V^p_s\cap Q$. Then choose $j\geq\ell$ big enough so that $[q^1_s\upharpoonright j]\cap[q^0_s\upharpoonright j]=\varnothing$ for every $s\in{}^{n}2$. Now simply extend $p$ to a condition $p':{}^{\leq n+1}2\longrightarrow Q\times\BB$ by defining $p'(s^\frown\varepsilon)=(q^\varepsilon_s,U^\varepsilon_s)$ for every $s\in{}^{n}2$ and $\varepsilon\in 2$, where each $U^\varepsilon_s\in\BB$ is such that $q^\varepsilon_s\in U^\varepsilon_s\subseteq V^p_s\cap [q^\varepsilon_s\upharpoonright j]$. It is easy to realize that $p'\in D^\textrm{ref}_\ell$.

For any fixed $\sigma=\{x_1,\ldots,x_k\}\in [\FF]^{<\omega}$ and $\ell\in\omega$, define
$$
D_{\sigma,\ell}=\{p\in\PPP:\textrm{there exists }i\in\omega\setminus\ell\textrm{ such that }
$$
$$
x(i)=x_1(i)=\cdots=x_k(i)=1\textrm{ for all }x\in U^p_s\textrm{ for all }s\in{}^{n_p}2\}.
$$
Let us check that each $D_{\sigma,\ell}$ is dense. Given $p\in\PPP$, $\sigma$ and $\ell$ as above, let $n=n_p$ and $q^0_s=q^p_s$ for every $s\in{}^{n}2$. Notice that
$$
\bigcap_{s\in{}^{n}2}q^p_s\cap\bigcap\sigma
$$
is an infinite subset of $\omega$, because $Q\subseteq\FF$ by assumption. So there exists $i\in\omega$ with $i\geq\ell$ such that 
$$
q^p_s(i)=x_1(i)=\cdots=x_k(i)=1
$$ 
for every $s\in{}^{n}2$. Since $Q$ has no isolated points, it is possible, for every $s\in{}^{n}2$, to choose $q^1_s\neq q^0_s$ such that $q^1_s\in V^p_s\cap [q^p_s\upharpoonright (i+1)]\cap Q$. Then choose $j\geq i+1$ big enough so that $[q^1_s\upharpoonright j]\cap[q^0_s\upharpoonright j]=\varnothing$ for every $s\in{}^{n}2$. Now simply extend $p$ to a condition $p':{}^{\leq n+1}2\longrightarrow Q\times\BB$ by defining $p'(s^\frown\varepsilon)=(q^\varepsilon_s,U^\varepsilon_s)$ for every $s\in{}^{n}2$ and $\varepsilon\in 2$, where each $U^\varepsilon_s\in\BB$ is such that $q^\varepsilon_s\in U^\varepsilon_s\subseteq V^p_s\cap [q^\varepsilon_s\upharpoonright j]$. It is easy to realize that $p'\in D_{\sigma,\ell}$.

Since $|\FF|<\cccc$, the collection of dense sets
$$
\DD=\{D_\ell:\ell\in\omega\}\cup\{D^\textrm{ref}_\ell:\ell\in\omega\}\cup\{D_{\sigma,\ell}:\sigma\in [\FF]^{<\omega},\ell\in\omega\}
$$
has also size less than $\cccc$. Therefore, by $\mathrm{MA(countable)}$, there exists a $\DD$-generic filter $G\subseteq\PPP$. Let $g=\bigcup G:\ct\longrightarrow Q\times\BB$. Given $s\in\ct$, pick any $p\in G$ such that $s\in\dom(p)$ and set $U_s=U^p_s$. For any $x\in 2^\omega$, since $\Sigma$ is a winning strategy for II, we must have $\bigcap_{n\in\omega}U_{x\upharpoonright n}\neq\varnothing$. Using the dense sets $D^\textrm{ref}_\ell$, one can easily show that such intersection is actually a singleton. Therefore, letting $f(x)$ be the unique element of $\bigcap_{n\in\omega}U_{x\upharpoonright n}$ yields a well-defined function $f:2^\omega\longrightarrow A$. Using condition ($\ref{disjointness}$) in the definition of $\PPP$, one sees that $f$ is injective. 

Next, we will show that $f$ is continuous in the standard topology, hence a homemorphic embedding by compactness. Fix $x\in 2^\omega$ and let $y=f(x)$. Fix $\ell\in\omega$. Since $G$ is a $\DD$-generic filter, there must be $p\in D^\textrm{ref}_\ell\cap G$. Let $n=n_p$. Notice that this implies $U_{x\upharpoonright n}=U^p_{x\upharpoonright n}\subseteq [y\upharpoonright\ell]$, hence $f(x')\in [y\upharpoonright\ell]$ whenever $x'\in [x\upharpoonright n]$.

Therefore $P=\ran(f)$ is a perfect subset of $A$. Finally, using the dense sets $D_{\sigma,\ell}$ one can show that $\FF\cup\{\bigcap P\}$ has the finite intersection property.
\end{proof}

\begin{corollary}\label{closedpsp} Assume that $\mathrm{MA(countable)}$ holds. Then there exists a non-principal ultrafilter $\UU\subseteq 2^\omega$ such that every closed subset of $\UU$ has the perfect set property.
\end{corollary}

\begin{question} Can the assumption that $\mathrm{MA(countable)}$ holds be dropped in Theorem $\ref{analyticpsp}$?
\end{question}

Observe that if $Q\subseteq 2^\omega$ is homeomorphic to $\QQQ$ in the standard topology, $A=\cl(Q)$ and $\TT_A$ is the topology obtained by declaring $A$ open, then $(\TT_A,Q,A)$ is an $\mathrm{A}$-triple because $\TT_A$ is Polish (see Lemma 13.2 in \cite{kechris}). It follows easily that the ultrafilter constructed in Theorem $\ref{analyticpsp}$ cannot contain closed copies of the rationals, hence it is completely Baire by Lemma $\ref{hurewicz}$.

\begin{question} Is an ultrafilter $\UU\subseteq 2^\omega$ such that $A\cap\UU$ has the perfect set property whenever $A$ is an analytic subset of $2^\omega$ necessarily completely Baire?
\end{question}

We also remark that if $\Gamma\subseteq\PP\bigl(2^\omega\bigr)$ is closed under $c$ and $\UU$ is such that $A\cap\UU$ has the perfect set property for all $A\in\Gamma$, then $A\setminus\UU$ has the perfect set property for all $A\in\Gamma$.

\section{Extending the perfect set property}

Assuming $V=L$, there exists an uncountable co-analytic set $A$ that does not contain any perfect set (see Theorem 25.37 in \cite{jech}). It follows that $\mathrm{MA(countable)}$ is not enough to extend Theorem $\ref{analyticpsp}$ to all co-analytic sets. This section is devoted to attaining a positive result for the co-analytic case. Actually, we will obtain a much stronger result (see Theorem \ref{lrpspu}). We will need a modest large cardinal assumption, a larger fragment of MA, and the negation of CH.

\begin{lemma}\label{pal2psp}
Assume that $\UU\subseteq 2^\omega$ is a $\mathrm{P}_{\omega_2}$-point. If $A\subseteq 2^{\omega}$ is such that every closed subspace of $A$ has the perfect set property, then $A\cap\UU$ has the perfect set property.
\end{lemma}
\begin{proof}
Let $A$ be as above, and assume that $A\cap\UU$ is uncountable.
Choose $B\subseteq A\cap\UU$ such that $|B|=\omega_1$. Since $\UU$ is a $\mathrm{P}_{\omega_2}$-point, there is a pseudointersection $x$ of $B$ in $\UU$. For some $n\in\omega$, uncountably many elements of $B$ are in the closed set $C=(x\setminus n)\!\uparrow$. By hypothesis, $A\cap C$ contains a perfect set $P$. We now have $A\cap\UU\supseteq P$ as desired.
Thus, $A\cap\UU$ has the perfect set property.
\end{proof}

It is not hard to verify that the hypothesis on $A$ in the above lemma is optimal. Let $x_0$ and $x_1$ be complementary infinite subsets of $\omega$. Identify each $\PP(x_i)$ with the perfect set $\{x\in 2^\omega: x(n)=0\textrm{ for all }n\in x_{1-i}\}$. Fix a Bernstein subset $B_i$ of $\PP(x_i)$ and set $A_i=B_i\cup\PP(x_{1-i})$ for each $i\in 2$. Each $A_i$ has the perfect set property. However, if $\UU\subseteq 2^\omega$ is an ultrafilter, then some $A_i\cap\UU$ lacks the perfect set property. Indeed, if $x_i\in\UU$, then $y\in\UU$ for some subset $y\subseteq x_i$ such that $x_i\setminus y$ is infinite. The perfect set $y\!\uparrow\cap\PP(x_i)$ contains $\cccc$ many elements of $B_i$, so $A_i\cap\UU$ has size $\cccc$ as well. However, $A_i\cap\UU\subseteq B_i$, so $A_i\cap\UU$ does not contain a perfect set.

The following lemma is essentially due to Ihoda (Judah) and Shelah (see Theorem 3.1 in~\cite{ihodashelah}). Given a class $\Gamma$, we define $\mathrm{PSP}(\Gamma)$ to mean that every $X\in\Gamma\cap\PP(2^\omega)$ has the perfect set property.

\begin{lemma}\label{mahloma}
The existence of a Mahlo cardinal is equiconsistent with
$$\mathrm{MA(}\sigma\text{-}\mathrm{centered)}+\neg\mathrm{CH}+\mathrm{PSP}(L(\RRR)).$$
\end{lemma}
\begin{proof}
Any generic extension by the Levy collapse $\mathrm{Col}(\omega,\kappa)$ of an inaccessible cardinal $\kappa$ to $\omega_1$ satisfies $\mathrm{PSP}(L(\RRR))$ (see the proof of Theorem 11.1 in~\cite{kanamori}). By the proof of Lemma 1.1 in~\cite{ihodashelah}, if $\kappa$ is inaccessible and $\PPP$ is a forcing poset that satisfies the following conditions, then every generic extension $V[G]$ of $V$ by $\PPP$ is such that $L(\RRR)^{V[G]}=L(\RRR)^{V[H]}$ for some $V$-generic filter $H\subseteq\mathrm{Col}(\omega,\kappa)$.
\begin{enumerate}
\item\label{kcc} $\PPP$ has the $\kappa$-cc. 
\item\label{komega1} $\PPP$ forces $\kappa=\omega_1$.
\item\label{kfactors} For every $R\subseteq\PPP$ of size less than $\kappa$, there exists $Q\subseteq\PPP$ such that $|Q|<\kappa$, $R\subseteq Q$, and $Q$ is completely embedded in $\PPP$ by the inclusion map.
\end{enumerate}
Assuming that there exists a Mahlo cardinal $\kappa$, the proof of Theorem 3.1 in \cite{ihodashelah} constructs a generic extension $V[G]$ of $V$ by a forcing $\PPP$ such that $V[G]$ satisfies $\mathrm{MA(}\sigma\text{-}\mathrm{centered)}+\neg\mathrm{CH}$, using a forcing poset $\PPP$ that satisfies conditions ($\ref{kcc}$), ($\ref{komega1}$) and ($\ref{kfactors}$). Therefore, $\mathrm{PSP}(L(\RRR))$ also holds in $V[G]$.

Conversely, $\mathrm{PSP}(L(\RRR))$ implies that all
injections of $\omega_1$ into $2^{\omega}$ are outside of $L(\RRR)$,
which in turn implies $\omega_1^{L[r]}<\omega_1$ for all reals $r$.
The proof of Theorem 3.1 in~\cite{ihodashelah} shows that 
if $\mathrm{MA(}\sigma\text{-}\mathrm{centered)}+\neg\mathrm{CH}$ holds and $\omega_1^{L[r]}<\omega_1$ for all reals $r$, then $\omega_1$ is Mahlo in $L$.
\end{proof}

For the convenience of the reader, we include the proof of the following standard lemma.

\begin{lemma}\label{pcpoint}
Assume that $\mathrm{MA(}\sigma\text{-}\mathrm{centered)}$ holds. Then there exists a $\mathrm{P}_\cccc$-point $\UU$.
\end{lemma}
\begin{proof}
Enumerate all subsets of $\omega$ as $\{z_\eta:\eta\in\cccc\}$. We will construct $\FF_\xi$ for every $\xi\in\cccc$ by transfinite recursion. By induction, we will make sure that the following requirements are satisfied.
\begin{enumerate}
\item $\FF_\mu\subseteq\FF_\eta$ whenever $\mu\leq\eta<\cccc$.
\item\label{fippc} $\FF_\xi$ has the finite intersection property for every $\xi\in\cccc$.
\item $|\FF_\xi|<\cccc$ for every $\xi\in\cccc$.
\item\label{ultrapc} By stage $\xi=\eta+1$, we must have decided whether $z_\eta\in\UU$: that is, $z_\eta^\varepsilon\in\FF_\xi$ for some $\varepsilon\in 2$.
\item\label{pseudopc} At stage $\xi=\eta+1$, we will make sure that $\FF_\xi$ contains a pseudointersection of $\FF_\eta$.
\end{enumerate}
Start by letting $\FF_0=\FR$. Take unions at limit stages. At a successor stage $\xi=\eta+1$, assume that $\FF_\eta$ is given. 

Since $\mathrm{MA(}\sigma\text{-}\mathrm{centered)}$ implies $\pppp=\cccc$ (see Theorem 7.12 in \cite{blass}), there exists an infinite pseudointersection $x$ of $\FF_\eta$. Now simply set $\FF_\xi=\FF_\eta\cup\{x,z_\eta^\varepsilon\}$ for a choice of $\varepsilon\in 2$ that is compatible with condition $(\ref{fippc})$.

In the end, let $\UU=\bigcup_{\xi\in\cccc}\FF_\xi$. Notice that $\UU$ will be an ultrafilter by ($\ref{ultrapc}$). Since $\pppp=\cccc$ is regular (see Theorem 7.15 in \cite{blass}), condition ($\ref{pseudopc}$) implies that $\UU$ is a $\mathrm{P}_\cccc$-point.
\end{proof}

It is well-known that  $\mathrm{MA(countable)}$ is not a sufficient hypothesis for the above lemma. Consider the Cohen model $W=V[(c_\alpha:\alpha<\omega_2)]$, where $V\vDash\mathrm{CH}$ and
each $c_\alpha$ is an element of $2^\omega$ that avoids all meager Borel sets with Borel codes in $V[(c_\beta:\beta<\omega_2, \beta\neq\alpha)]$. Observe that every $x\in 2^\omega$ is in $V[(c_\alpha:\alpha\in I)]$ for some countable set $I\subseteq\omega_2$. In this model, $\textrm{cov}(\MM)=\cccc=\omega_2$, so $\mathrm{MA(countable)}+\neg\mathrm{CH}$ holds (see Theorem 7.13 in \cite{blass}). However, if $\UU\in W$ is a non-principal ultrafilter, then $\UU\cap V[(c_\alpha:\alpha<\omega_1)]$ is a subset of $\UU$ of size $\omega_1$ with no infinite pseudointersection.

\begin{theorem}\label{lrpspu}
It is consistent, relative to a Mahlo cardinal, that there exists a non-principal ultrafilter $\UU\subseteq 2^\omega$ such that $A\cap\UU$ has the perfect set property for all $A\in\PP\bigl(2^\omega\bigr)\cap L(\RRR)$. On the other hand, if there exists such an ultrafilter $\UU$, then $\omega_1$ is inaccessible in $L$.
\end{theorem}
\begin{proof}
Assume that $\mathrm{MA(}\sigma\text{-}\mathrm{centered)}+\neg\mathrm{CH}+\mathrm{PSP}(L(\RRR))$ holds, which is consistent relative to a Mahlo cardinal by Lemma~\ref{mahloma}. By Lemma~\ref{pcpoint}, there exists a $\mathrm{P}_{\cccc}$-point $\UU$. Since $\neg\mathrm{CH}$ holds, $\UU$ is a $\mathrm{P}_{\omega_2}$-point. Fix $A\in\PP\bigl(2^\omega\bigr)\cap L(\RRR)$. Every closed subspace $C$ of $A$ is also in $L(\RRR)$ because $C=A\cap[T]$ for some tree $T\subseteq 2^{<\omega}$. By $\mathrm{PSP}(L(\RRR))$, all such $C$ have the perfect set property. So $A\cap\UU$ has the perfect set property by Lemma~\ref{pal2psp}.

For the second half of the theorem, assume that $\UU\subseteq 2^\omega$ is a non-principal ultrafilter such that $A\cap\UU$ has the perfect set property for all $A\in\PP\bigl(2^\omega\bigr)\cap L(\RRR)$. First, observe that given $A$ as above, $c[A]$ is in $L(\RRR)$ too, so $A\cap\UU$ and $c[A]\cap\UU$ have the perfect set property. Since
$$
A=(A\cap\UU)\cup(A\cap c[\UU])=(A\cap\UU)\cup c[c[A]\cap\UU],
$$
it follows that $A$ itself has the perfect set property. So $\mathrm{PSP}(L(\RRR))$ holds, which implies $\omega_1^{L[r]}<\omega_1$ for all reals $r$. Therefore $\omega_1$ is inaccessible in $L$.
\end{proof}

\begin{question}
What is the exact consistency strenght of a non-principal ultrafilter $\UU\subseteq 2^\omega$ such that $A\cap\UU$ has the perfect set property for all $A\in\PP\bigl(2^\omega\bigr)\cap L(\RRR)$? In particular, does the Levy collapse $\mathrm{Col}(\omega,\kappa)$ of an inaccessible cardinal $\kappa$ to $\omega_1$ force such an ultrafilter?
\end{question}

\section{$\mathrm{P}$-points}

Given a non-principal ultrafilter $\UU\subseteq 2^\omega$, it seems natural to investigate whether there is any relation between the topological properties of $\UU$ that we studied so far and combinatorial properties of $\UU$. In order to construct several kinds of non-$\mathrm{P}$-points, we will essentially use an idea from \cite{kunen}.

\begin{definition}
A \emph{mixed independent family} is a pair $(\FF,\Aa)$ of collections of subsets of $\omega$ such that
$$
\bigcap\sigma\cap\bigcap_{x\in\tau}x^{w(x)}
$$
is infinite whenever $\sigma\in[\FF]^{<\omega}$, $\tau\in[\Aa]^{<\omega}$ and $w:\tau\longrightarrow 2$. A \emph{dual mixed independent family} is a pair $(\II,\BB)$ of collections of subsets of $\omega$ such that $(c[\II],c[\BB])$ is a mixed independent family.
\end{definition}

\begin{lemma}\label{nonP}
Let $(\FF,\Aa)$ be a mixed independent family such that $\Aa$ is infinite. Then there exists a non-$\mathrm{P}$-point $\UU$ extending $\FF\cup\Aa$.
\end{lemma}
\begin{proof}
Fix a countably infinite subset $\BB$ of $\Aa$. It is easy to check that 
$$
\GG=\FF\cup\Aa\cup\{\omega\setminus x:x\subseteq^\ast y\textrm{ for every }y\in\BB\}
$$
has the finite intersection property. Let $\UU$ be any ultrafilter extending $\GG$. It is clear that $\BB$ has no pseudointersection in $\UU$.
\end{proof}

Similarly, one can prove the following.

\begin{lemma}\label{dualnonP}
Let $(\II,\BB)$ be a dual mixed independent family such that $\BB$ is infinite. Then there exists a maximal ideal $\JJ$ extending $\II\cup\BB$ that is not a $\mathrm{P}$-ideal.
\end{lemma}

We will begin by studying the relation between $\mathrm{P}$-points and completely Baire ultrafilters.

\begin{theorem} There exists a non-$\mathrm{P}$-point $\UU\subseteq 2^\omega$ that is not completely Baire.
\end{theorem}
\begin{proof}
We will use the same notation as in the proof of Theorem~\ref{closed}.
Choose $C=\QQQ$, so that any ultrafilter extending $\GG$ 
will contain a closed copy of $\QQQ$. Now simply apply Lemma $\ref{nonP}$ to $(\varnothing,\GG)$.
\end{proof}

\begin{theorem}\label{cbPpoint} Assume that $\mathrm{MA(countable)}$ holds. Then there exists a $\mathrm{P}$-point $\UU\subseteq 2^\omega$ that is completely Baire.
\end{theorem}
\begin{proof}
Enumerate all countable collections of subsets of $\omega$ as $\{\CC_\eta:\eta\in\cccc\}$. The setup of the construction will be as in the proof of Theorem $\ref{cb}$, but we will do different things at even and odd successor stages.

Start by letting $\FF_0=\FR$. Take unions at limit stages. At a successor stage $\xi=2\eta+1$, assume that $\FF_{2\eta}$ is given, then take care of $Q_\eta$ as in the proof of Theorem $\ref{cb}$. At a successor stage $\xi=2\eta+2$, assume that $\FF_{2\eta+1}$ is given, then take care of $\CC_\eta$ as follows.

First assume that there exists $x\in\CC_\eta$ such that $\FF_{2\eta+1}\cup\{\omega\setminus x\}$ has the finite intersection property. In this case, we can just set $\FF_\xi=\FF_{2\eta+1}\cup\{\omega\setminus x\}$. Now assume that $\FF_{2\eta+1}\cup\{\omega\setminus x\}$ does not have the finite intersection property for any $x\in\CC_\eta$. It is easy to check that this implies $\CC_\eta\subseteq\langle\FF_{2\eta+1}\rangle$. Recall that $\mathrm{MA(countable)}$ implies $\dddd=\cccc$ (see, for example, Proposition 5.5 and Theorem 7.13 in \cite{blass}). So, by Proposition 6.24 in \cite{blass}, there exists a pseudointersection $x$ of $\CC_\eta$ such that $\FF_{2\eta+1}\cup\{x\}$ has the finite intersection property. Finally, set $\FF_\xi=\FF_{2\eta+1}\cup\{x\}$.
\end{proof}

\begin{question} For a non-principal ultrafilter $\UU\subseteq 2^\omega$, is being a $\mathrm{P}$-point equivalent to being completely Baire?
\end{question}

Now we turn to the relation between $\mathrm{P}$-points and countable dense homogeneous ultrafilters.

\begin{theorem} Assume that $\mathrm{MA(countable)}$ holds. Then there exists a non-principal ultrafilter $\UU\subseteq 2^\omega$ that is countable dense homogeneous but not a $\mathrm{P}$-point.
\end{theorem}
\begin{proof}
For notational convenience, we will actually construct a maximal ideal $\JJ\subseteq 2^\omega$ that is countable dense homogeneous but not a $\mathrm{P}$-ideal.

The setup of the construction will be as in the proof of Theorem $\ref{cdh}$, but we will simultaneously construct $\BB_\xi$ for $\xi\in\cccc$ so that the following conditions will be satisfied. In the end, set $\BB=\bigcup_{\xi\in\cccc}\BB_\xi$ and apply Lemma $\ref{dualnonP}$ to $(\II,\BB)$.
\begin{enumerate}
\item $\BB_\mu\subsetneq\BB_\eta$ whenever $\mu<\eta<\cccc$.
\item $(\II_\xi,\BB_\xi)$ is a dual mixed independent family for every $\xi\in\cccc$.
\item $|\BB_\xi|<\cccc$ for every $\xi\in\cccc$.
\end{enumerate}

Start by letting $(\II_0,\BB_0)=(\FIN,\varnothing)$. Take unions at limit stages. At a successor stage $\xi=\eta+1$, assume that $(\II_\eta,\BB_\eta)$ is given. First get $x$ by applying Lemma $\ref{keycdhmixed}$ with $\II=\II_\eta$, $\BB=\BB_\eta$ and $(D,E)=(D_\eta,E_\eta)$. Then, as in the proof of Lemma $\ref{independentncdh}$, use $\mathrm{MA(countable)}$ to get $y\notin\BB_\eta$ such that $(\II_\eta\cup\{x\},\BB_\eta\cup\{y\})$ is still a dual mixed independent family. Finally, set $(\II_\xi,\BB_\xi)=(\II_\eta\cup\{x\},\BB_\eta\cup\{y\})$.
\end{proof}

The following lemma is easily proved by modifying  the proof of Lemma $\ref{keycdh}$ (substitute the dense sets $D_{\sigma,\ell}$ with the obviously defined dense sets $D_{\sigma,\tau,w,\ell}$).

\begin{lemma}\label{keycdhmixed} Assume that $\mathrm{MA(countable)}$ holds. Let $(\II,\BB)$ be a dual mixed independent family such that $|\II|<\cccc$ and $|\BB|<\cccc$. Fix two countable dense subsets $D$ and $E$ of $2^\omega$ such that $D\cup E\subseteq\langle\II\rangle$. Then there exists a homeomorphism $f:2^\omega\longrightarrow 2^\omega$ and $x\in 2^\omega$ such that $f[D]=E$, $(\II\cup\{x\},\BB)$ is still a dual mixed independent family and $\{d+f(d):d\in D\}\subseteq x\!\downarrow$.
\end{lemma}

\begin{theorem} Assume that $\mathrm{MA(countable)}$ holds. Then there exists a non-$\mathrm{P}$-point $\UU\subseteq 2^\omega$ that is not countable dense homogeneous.
\end{theorem}
\begin{proof}
Let $\Aa$ be as in Lemma $\ref{independentncdh}$. By the proof of Theorem~\ref{ncdh}, no ultrafilter extending $\Aa$ is countable dense homogeneous. Now simply apply Lemma~\ref{nonP} to $(\varnothing,\Aa)$.
\end{proof}

\begin{theorem} Assume that $\mathrm{MA(countable)}$ holds. Then there exists a $\mathrm{P}$-point $\UU\subseteq 2^\omega$ that is countable dense homogeneous.
\end{theorem}
\begin{proof}
For notational convenience, we will actually construct a maximal ideal $\JJ\subseteq 2^\omega$ that is countable dense homogeneous and a $\mathrm{P}$-ideal. Enumerate all countable collections of subsets of $\omega$ as $\{\CC_\eta:\eta\in\cccc\}$. The setup of the construction will be as in the proof of Theorem $\ref{cdh}$, but we will do different things at even and odd successor stages.

Start by letting $\II_0=\FIN$. Take unions at limit stages. At a successor stage $\xi=2\eta+1$, assume that $\II_{2\eta}$ is given, then take care of $(D_\eta,E_\eta)$ as in the proof of Theorem $\ref{cdh}$. At a successor stage $\xi=2\eta+2$, assume that $\II_{2\eta+1}$ is given, then take care of $\CC_\eta$ as follows.

First assume that there exists $x\in\CC_\eta$ such that $\II_{2\eta+1}\cup\{\omega\setminus x\}$ has the finite union property. In this case, we can just set $\II_\xi=\II_{2\eta+1}\cup\{\omega\setminus x\}$. Now assume that $\II_{2\eta+1}\cup\{\omega\setminus x\}$ does not have the finite union property for any $x\in\CC_\eta$. It is easy to check that this implies $\CC_\eta\subseteq\langle\II_{2\eta+1}\rangle$. As in the proof of Theorem $\ref{cbPpoint}$, it is possible to get a pseudounion $x$ of $\CC_\eta$ such that $\II_{2\eta+1}\cup\{x\}$ has the finite union property. Finally, set $\II_\xi=\II_{2\eta+1}\cup\{x\}$.
\end{proof}

\begin{question} Is a $\mathrm{P}$-point $\UU\subseteq 2^\omega$ necessarily countable dense homogeneous?
\end{question}

Finally, we will investigate the relation between $\mathrm{P}$-points and the perfect set property.

\begin{theorem} There exists a non-$\mathrm{P}$-point $\UU\subseteq 2^\omega$ with a closed subset of cardinality $\cccc$ that does not have the perfect set property.
\end{theorem}
\begin{proof}
We will use the same notation as in the proof of Theorem~\ref{closed}.
Choose $C$ to be a Bernstein set in $2^\omega$, so that any ultrafilter extending $\GG$ will have a closed subset without the perfect property. Now simply apply Lemma $\ref{nonP}$ to $(\varnothing,\GG)$.
\end{proof}

\begin{theorem} Assume that $\mathrm{MA(countable)}$ holds. Then there exists a $\mathrm{P}$-point $\UU\subseteq 2^\omega$ such that $A\cap\UU$ has the perfect set property whenever $A$ is an analytic subset of $2^\omega$.
\end{theorem}
\begin{proof}
Enumerate all countable collections of subsets of $\omega$ as $\{\CC_\eta:\eta\in\cccc\}$. The setup of the construction will be as in the proof of Theorem $\ref{analyticpsp}$, but we will do different things at even and odd successor stages.

Start by letting $\FF_0=\FR$. Take unions at limit stages. At a successor stage $\xi=2\eta+1$, assume that $\FF_{2\eta}$ is given, then take care of $(\TT_\eta,A_\eta,Q_\eta)$ and $z_\eta$ as in the proof of Theorem $\ref{analyticpsp}$. At a successor stage $\xi=2\eta+2$, assume that $\FF_{2\eta+1}$ is given, then take care of $\CC_\eta$ as in the proof of Theorem $\ref{cbPpoint}$.
\end{proof}

\begin{question}\label{analyticpspppoint} For a non-principal ultrafilter $\UU\subseteq 2^\omega$, is being a $\mathrm{P}$-point equivalent to $A\cap\UU$ having the perfect set property whenever $A$ is an analytic subset of $2^\omega$?
\end{question}

\noindent Observe that Lemma $\ref{pal2psp}$ might be viewed as a partial answer to Question $\ref{analyticpspppoint}$.

\bigskip

\noindent\textbf{Acknowledgements.} The first author thanks Jan van Mill for making him notice Proposition $\ref{many}$ and the fact that the topology of ultrafilters as subspaces of $2^\omega$ was `unexplored territory'. Together with Question 3.2 from \cite{hrusak}, that was the main motivation for this article. The first author also thanks Kostas Beros for suggesting the second proof of Lemma $\ref{perfectindependent}$ and helping to shape the present proof of Lemma $\ref{keycdh}$. Both authors thank Ken Kunen for kindly `donating' Theorem $\ref{closed}$ to them, and the anonymous referee for valuable comments on an earlier version of this paper.

\end{document}